\documentclass{amsart}

\usepackage{amssymb, amsfonts,amsthm,amsmath,latexsym}
\usepackage[all]{xy}

\usepackage[T1]{fontenc}

\usepackage[utf8]{inputenc}
\usepackage[english]{babel}
\usepackage{amssymb}
\usepackage{amsthm}
\usepackage{graphics}
\usepackage{amsmath}
\usepackage{amstext}
\usepackage{multirow}
\usepackage{hyperref}
\usepackage{comment}
\usepackage{color}
\usepackage{tikz-cd}

\newtheorem{thm}{Theorem}[section]
\newtheorem{dfn}[thm]{Definition}
\newtheorem{lemma}[thm]{Lemma}
\newtheorem{prop}[thm]{Proposition}
\newtheorem{cor}[thm]{Corollary}
\newtheorem{que}{Question}
\newtheorem{conj}[que]{Conjecture}

\newtheorem{THM}{Theorem}

\theoremstyle{remark}
\newtheorem{ex}{Example}
\newtheorem{rem}[thm]{Remark}

\newcommand{\mb}{\mathbb}
\newcommand{\mc}{\mathcal}
\newcommand{\mf}{\mathfrak}
\newcommand{\R}{\mb R}
\newcommand{\C}{\mb C}
\newcommand{\Hj}{\mb H}
\newcommand{\Pj}{\mb P}
\newcommand{\Z}{\mb Z}
\newcommand{\Q}{\mb Q}

\newcommand{\D}{\mb D}

\newcommand{\inv}{^{-1}}
\newcommand{\F}{\mc F}
\newcommand{\G}{\mc G}

\newcommand{\rato}{\dashrightarrow}

\DeclareMathOperator{\SL}{SL}
\DeclareMathOperator{\PSL}{PSL}

\DeclareMathOperator{\Gal}{Gal}

\DeclareMathOperator{\im}{Im}
\DeclareMathOperator{\codim}{codim}

\DeclareMathOperator{\Aut}{Aut}

\DeclareMathOperator{\Bir}{Bir}
\DeclareMathOperator{\Psaut}{PsAut}

\DeclareMathOperator{\Sing}{Sing}

\DeclareMathOperator{\Aff}{Aff}
\DeclareMathOperator{\Ric}{Ric}
\DeclareMathOperator{\dev}{dev}
\DeclareMathOperator{\Exc}{Exc}
\DeclareMathOperator{\Bs}{Bs}
\DeclareMathOperator{\Id}{Id}

\numberwithin{equation}{section}

\begin{document}
\title{Symmetries of transversely projective foliations}
\author{F. Lo Bianco, E. Rousseau, F. Touzet}
$$ $$
\maketitle

\begin{abstract}
Given a (singular, codimension $1$) holomorphic foliation $\F$ on a complex projective manifold $X$, we study the group $\Psaut(X,\F)$ of pseudo-automorphisms of $X$ which preserve $\F$; more precisely, we seek sufficient conditions for a finite index subgroup of $\Psaut(X,\F)$ to fix all leaves of $\F$. It turns out that if $\F$ admits a (possibly degenerate) transverse hyperbolic structure, then the property is satisfied; furthermore, in this setting we prove that all entire curves are algebraically degenerate. We prove the same result in the more general setting of transversely projective foliations, under the additional assumptions of non-negative Kodaira dimension and that for no generically finite morphism $f\colon X'\to X$ the foliation $f^*\F$ is defined by a closed rational $1$-form.
\end{abstract}

\section{Introduction}

In this article we study the symmetries of holomorphic foliations, i.e. automorphisms (or birational transformations) of the ambient manifold which send each leaf to another leaf; we denote by $\Aut(X,\F)$ the group of such automorphisms. In particular, we focus on the following question:

\begin{que}
\label{main question}
Under which conditions does a finite index subgroup of $\Aut(X,\F)$ preserve each leaf of $\F$?
\end{que}

If the above condition is satisfied, we will say that the \emph{transverse action} of $\Aut(X,\F)$ (on $\F$) is finite.

\begin{ex}
Let $\F$ be a linear foliation on a compact complex torus $X=\C^n/\Lambda$. Then the group $\Aut(X,\F)$ contains the group of translations of $X$, and in particular its transverse action is infinite.
\end{ex}

\begin{ex}
Since the group of automorphisms of a projective variety of general type $X$ is finite, so is the transverse action of $\Aut(X,\F)$ for any foliation $\F$ on $X$. \\
By taking the pull-back foliation on a product $X\times Y$ ($Y$ being for example a compact torus) one obtains a foliation with an infinite group of symmetries which has finite transverse action.
\end{ex}

\subsection{Main results}

From now on we suppose that $X$ is a complex projective manifold and that $\F$ is a (possibly singular) foliation of codimension $1$. Recall that a birational transformation $f\colon X\rato X$ is called a \emph{pseudo-automorphism} if $f$ induces an isomorphism $U\cong V$ between two Zariski-open sets such that $\codim(X\setminus U),\codim(X\setminus V)\geq 2$; or, equivalently, if $f$ and $f\inv$ do not contract any hypersurface.

We say that $\F$ admits a \emph{transverse hyperbolic structure} if, roughly speaking, outside a degeneracy divisor $H\subset X$ the foliation admits local first integrals $F_i\colon U_i \to \mb D$ which are uniquely defined up to left composition with automorphisms of $\mb D$; see Definition \ref{def hyperbolic}.\\
The third-named author showed in \cite{MR3124741}  that $\F$ admits a transverse hyperbolic structure if the conormal bundle $N^*_\F$ is pseudo-effective and the positive part of its Zariski decomposition is non-trivial; see Remark \ref{hyp->pseff}.\\
We denote by $\Psaut(X,\F)$ the group of pseudo-automorphisms of $X$ which preserve $\F$.

In this context, we prove that the foliation is essentially the pull-back of a foliation on a projective variety of general type, which implies the transverse finiteness of the action of $\Psaut(X,\F)$; furthermore, we obtain a result on entire curves on $X$:

\begin{THM}
\label{thm transv hyp}
Let $X$ be a projective manifold and let $\F$ be a transversely hyperbolic codimension $1$ foliation. Then
\begin{itemize}
\item there exists a generically finite morphism $\pi\colon X'\to X$, a morphism $\psi\colon X'\to B$ onto a projective variety $B$ of general type and a foliation $\G$ on $B$ such that $\pi^*\F=\psi^*\G$;
\item the transverse action of $\Psaut(X,\F)$ is finite;
\item any entire curve $f: \C \to X$ is algebraically degenerate i.e. $f(\C)$ is not Zariski dense.
\end{itemize}
\end{THM}

For a proof, see Theorem \ref{hyperbolic transverse action} and Theorem \ref{degen}. This result should be seen as a generalization of well-known properties of hyperbolic curves. It is also important to remark that such a statement is wrong in the non-K\"ahler setting as we will see in the striking example of Inoue surfaces.

Transversely hyperbolic foliations are a special case of \emph{transversely projective} foliations: in this case, roughly speaking, the distinguished first integrals have values in $\Pj^1$ and they are uniquely defined up to left composition with automorphisms of $\Pj^1$ (see Definition \ref{def projective}). In this context the description is less precise, and we are forced to introduce a dichotomy:

\begin{THM}
\label{thm transv proj}
Let $X$ be a projective  manifold with $\kappa(X)\geq 0$ and $\F$ be a transversely projective (possibly singular) foliation of codimension $1$ on $X$. 
Then 
\begin{itemize}
\item either there exists a generically finite  morphism $\pi \colon X'\to X$ such that $\pi^* \F$ is defined by a closed rational $1$-form;
\item or the transverse action of $\Psaut(X,\F)$ is finite.
\end{itemize}
\end{THM}

Remark that the first alternative contains the case of algebraically integrable foliations.

The proofs of the results follow the same overall strategy, although in the general case of transversely projective foliations one needs to address some additional technical difficulties:
\begin{itemize}
\item we apply a result of Corlette-Simpson \cite{MR2457528} which allows to factor (see Definition \ref{def factor}) the monodromy of the structure either through a curve or through a quotient of the polydisk $\mb D^N/\Gamma$ (the transverse hyperbolic and transverse projective cases are treated in detail in \cite{MR3644247} and  \cite{MR3522824} respectively);
\item the case of curves can be treated almost by hand (in the case of a projective structure, we use a classification result of Cantat and Favre \cite{MR1998612});
\item for the case of quotients of the polydisk, we apply a result of Brunebarbe \cite{brunebarbe2016strong}, which ensures that the image of the morphism $\psi\colon X\rato \mb D^N/\Gamma$ is of (log-)general type, and in particular its group of pseudo-automorphisms is finite;
\item one shows that $\psi$ is essentially equal to the Shafarevich morphism of the monodromy representation, hence it is invariant by $\Psaut(X,\F)$; then one can restrict to fibres (in the transverse projective case, one needs to apply \cite{lobiancopadic}, hence the assumption on the Kodaira dimension).
\end{itemize}

\subsection{A conjecture}

In the context of fibrations (i.e. algebraically integrable foliations), Question \ref{main question} was studied by the first-named author in \cite{lobiancopadic}. Theorem A in \emph{loc-cit.} suggests the following conjecture:

\begin{conj}
\label{main conjecture}
Let $X$ be a projective manifold such that $\kappa(X)\geq 0$, $\F$ be a foliation on $X$ and $L$ be a line bundle on $X$. Suppose that $L$ admits a singular hermitian metric whose curvature form defines, up to sign, a  transverse hermitian metric on $\F$.\\
Then a birational transformation of $X$ preserving $\F$ and $L$ has transversely finite action.
\end{conj}

Here, by a transverse hermitian metric we mean a (semi-)positive closed $(1,1)$-current, which is invariant by the holonomy of $\F$ and which induces a smooth hermitian metric on the normal bundle $N_\F$ in codimension 1 (indeed, outside the singular locus of $\F$); this is also Mok's definition of a semi-k\"ahler structure \cite[Definition 1.2.1]{AIF_2000__50_2_633_0}. The third-named author showed in \cite{MR3403731} that, if $\F$ has codimension $1$ and is regular, the existence of a closed positive and holonomy invariant current without atomic part   
implies the existence of a such a transverse hermitian metric. Moreover, the latter can be chosen to be \emph{homogeneous} (i.e. hyperbolic, euclidean or spherical, depending on the sign of the curvature tensor).

\begin{rem}
By the results proven in the forthcoming Section \ref{sec:hyperbolic}, it seems rather natural to state the same conjecture under more general  assumptions on the transverse metric inherited from the curvature current of $L$ (allowing for instance weaker regularity and additional  degenaracies along invariant hypersurfaces).
\end{rem}

\subsection{Structure of the text}

In Section \ref{sec:preliminaries} we present the formal definitions of transversely hyperbolic and projective structures, and give the interpretation of these definitions in terms of developing maps and monodromy; we also briefly recall some of the properties of Shimura modular orbifolds which will be used later, as well as the definition of factorization of a representation and a result of lifting of pseudo-automorphisms to finite \'etale covers. In Section \ref{sec:hyperbolic} and \ref{sec transv proj} we prove Theorem \ref{thm transv hyp} and \ref{thm transv proj} respectively; we also show that Theorem \ref{thm transv hyp} cannot be extended to the general (non-K\"ahler) compact case. Finally, in Section \ref{sec tori} we describe the symmetries of codimension $1$ foliations on compact complex tori; in particular, we show that Conjecture \ref{main conjecture} is (trivially) satisfied in this case.


\section{Preliminaries}
\label{sec:preliminaries}

\subsection{Transverse structures on codimension $1$ foliations}

Throughout this section, we denote by $X$ a complex (projective) manifold and by $\F$ a codimension $1$ (possibly singular) foliation. By a (smooth) transverse structure on $\F$ we mean, roughly speaking, a geometric structure (in a broad sense: metric, homogeneous structure...) defined on the normal bundle $N_\F$ which is invariant by the holonomy of $\F$.\\
Of course we need to specify the behavior at singular points of $\F$; furthermore, we will consider more generally singular transverse structure, which may degenerate (in a prescribed way) along an $\F$-invariant hypersurface $H$.

\subsubsection{Definitions}
Let us start with the formal definition of transverse hyperbolic structure, see \cite{MR3403731}. \\
\emph{Caution! We use a different notation than \cite{MR3403731}, where the metric and the associated curvature current are denoted by $\eta_T$ and $-T$ respectively.} 

\begin{dfn}
\label{def hyperbolic}
A \emph{(branched) transverse hyperbolic structure} on $\F$ is the datum of a non-trivial positive closed $(1,1)$-current $T$ such that:
\begin{itemize}
\item $T$ is invariant by the holonomy of $\F$ (or simply $\F$-invariant), meaning that, if $\omega$ is a local holomorphic $1$-form defining $\F$, we have $\omega \wedge T=0$;
\item $T$ induces a singular hermitian metric on $N_\F$ (in the sense of Demailly, see \cite{MR1178721});
\item if $\Theta_T$ denotes the curvature current associated to $T$, we have $\Theta_T=-(T + [N])$, where $[N]$ denotes the current of integration along a $\Q$-effective divisor $N$.
\end{itemize}
The hypersurface $H:=Supp(N)$ is the degeneracy locus of the transverse hyperbolic structure.
\end{dfn}

A closed non-trivial semi-positive current satisfying the first and second condition is called a \emph{ singular transverse metric} of $\F$. It can be then locally written as $T=ie^{2\psi}\omega\wedge\bar\omega$ where $\omega$ is a local closed one-form defining $\F$ and $\psi$ is $L_{loc}^1$. If $x\in X$ is a regular point of $\F$, we can describe the foliation by a local coordinate $dz=0$, so that locally
$$T=ie^{2\psi(z)}dz\wedge d\bar z,$$

The associated curvature current is then locally defined as 
$$\Theta_T=-\frac {i} {\pi} \partial \bar\partial \psi.$$

A transversely (branched) euclidean (respectively, spherical) structure is defined in an analogous way by imposing that $\Theta_T=-[N]$ (respectively, $\Theta_T=T-[N]$). 

\begin{rem}
\label{hyp->pseff}
If a foliation $\F$ admits a transverse hyperbolic structure, then $N^*_\F$ is pseudo-effective. Indeed, a positive, holonomy invariant current $T$ defining the hyperbolic structure defines a singular hermitian metric on $N_\F$; its curvature form, which is equal to $-T-[N]$, represents the class $c_1(N_\F)$. Therefore, the class $c_1(N^*_\F)=-c_1(N_\F)$ is represented by the positive current $T+[N]$, meaning that $N^*_\F$ is pseudo-effective.

Conversely, the third-named author showed in \cite[Theorem 1]{MR3124741} that, if $N^*_\F$ is pseudo-effective, then $\F$ admits
\begin{itemize}
\item either a transverse hyperbolic structure,
\item or a transverse euclidean structure.
\end{itemize}

Moreover $N$ can be chosen to coincide with the negative part of the Zariski decomposition of $c_1(N^*_\F)$ (see \cite{MR2050205}). In this situation, the first part of the alternative exactly occurs when the positive part is non-trivial and the structures are then unique.
\end{rem}

Instead of considering local first integral with values in $\mb D$, one can pick more generally first integrals with values in $\Pj^1$, well-defined up to automorphisms of $\Pj^1$. In order to define a projective structure (see \cite{MR2337401, MR3522824}), one imposes the following conditions on the singular locus (i.e. the hypersurface where the structure degenerates):

\begin{dfn}
\label{def projective}
A \emph{transverse projective structure} on $\F$ is the data of a triple $(E, \nabla, \sigma)$ where
\begin{itemize}
\item $E$ is a rank $2$ vector bundle;
\item $\nabla$ is a flat meromorphic connection on $E$;
\item $\sigma \colon X \rato E$ is a meromorphic section of $\Pj(E)\rato X$ such that, if $\mc R$ denotes the Riccati foliation on $\Pj(E)$ determined by (the projectivization of) $\nabla$, $\F=\sigma^*\mc R$.
\end{itemize}
Such triples are considered modulo a natural relation of birational equivalence (see \cite{MR3522824}).
\end{dfn}

As explained in \cite[\textsection 6.1]{MR3644247}, transversely hyperbolic foliations are a special case of transversely projective foliations.

\subsubsection{Distinguished first integrals and monodromy representation}
Let $\F$ be a transversely hyperbolic foliation on a manifold $X$; denote by $H\subset X$ the polar hypersurface of $T$. Remark that, locally at points of $X_0:=X\setminus H$, $\F$ can be defined by a local first integral
$$F_i\colon U_i\to \mb D$$
which are uniquely defined modulo composition to the left by elements of $Isom(\mb D)=\Aut(\mb D)=\PSL_2(\R)$.\\
Following such distinguished first integrals along closed paths yields a developing map
$$\dev \colon \widetilde X_0 \to \mb D,$$
where $\widetilde X_0$ denotes the universal cover of $X_0$, and a monodromy representation
$$\rho\colon \pi_1(X_0)\to \PSL_2(\R)$$
such that
$$\rho(\gamma) \circ\dev=\dev \circ \gamma \qquad \text{for all }\gamma\in \pi_1(X_0).$$
Here, we identify $\pi_1(X_0)$ with the group of deck transformations of the universal cover $\widetilde X_0 \to X_0$.

Similarly, if $\F$ is a transversely projective foliation and $H \subset X$ denotes the polar hypersurface of the connection $\nabla$, locally at points of $X_0:=X\setminus H$ the foliation $\F$ admits distinguished (meromorphic) first integrals
$$F_i\colon U_i \rato \Pj^1$$
which are uniquely defined modulo composition to the left by elements of $\Aut(\Pj^1)=\PSL_2(\C)$.\\
Following such distinguished first integrals along closed paths yields a (meromorphic) developing map
$$\dev \colon \widetilde X_0 \rato \mb \Pj^1,$$
where $\widetilde X_0$ denotes the universal cover of $X_0$, and a monodromy representation
$$\rho\colon \pi_1(X_0)\to \PSL_2(\C)$$
such that
$$\rho(\gamma) \circ\dev=\dev \circ \gamma \qquad \text{for all }\gamma\in \pi_1(X_0).$$

\subsubsection{Singularities of transverse projective structures}
Let us conclude the introduction to transverse projective (or hyperbolic) structures by a brief discussion on the singular locus $H$ introduced above.

\begin{dfn}
We say that a transversely projective foliation has regular singularities if the corresponding connection has at worst regular singularities in the sense of \cite{MR0417174}.
\end{dfn}

As remarked in \cite[\textsection 6.1]{MR3644247}, transversely hyperbolic foliations have regular singularities when considered as transversely projective foliations.

For the purposes of this article, one can simply keep in mind the following property: if a transversely projective foliation has regular singularities and the monodromy (of a small loop) around an irreducible hypersurface $D\subset H$ is trivial, then a distinguished first integral defined in a neighborhood of $D$ extends meromorphically through $D$.

\subsection{Shimura modular orbifolds}

Recall that an \emph{orbifold} is Hausdorff topological space which is locally modelled on finite quotients of $\C^n$. One defines an orbifold cover as a map $f\colon X\to Y$ between orbifolds which is locally conjugated to a quotient map
$$\C^n/\Gamma_0 \to \C^n/\Gamma_1 \qquad \Gamma_0 \leq \Gamma_1.$$
Then one can see that given an orbifold $X$ there exists a universal orbifold cover $\pi\colon \widetilde X \to X$; the \emph{orbifold fundamental group} $\pi_1^{orb}(X)$ is then defined as the group of deck transformations of $\pi$.\\
For example, if $U$ is a simply connected complex manifold and $G\leq \Aut(U)$ is a discrete subgroup such that the stabilizer of each point of $U$ is finite, then the quotient $X=U/G$ admits a natural orbifold structure such that $\pi_1^{orb}(X)=G$.

Following Corlette and Simpson \cite{MR2457528} (see also \cite{MR3522824} and references therein), a polydisk Shimura modular orbifold is a  quotient $\mathfrak H$ of  a polydisk $\mathbb D^n$ by a  group of the form $U(P,\Phi)$ where
$P$ is a projective module of  rank two over the ring of integers $\mathcal O_L$ of a  totally imaginary  {quadratic extension}
$L$ of a totally real number field $F$; $\Phi$ is a skew hermitian form on $P_L=P \otimes_{\mathcal O_L} L$; and $U(P,\Phi)$  is the subgroup of the $\Phi$-unitary group  $U(P_L,\Phi)$
consisting of elements which preserve $P$. This group acts naturally  on $\mathbb D^n$ where $n$ is half the number of
embeddings $\sigma : L \to \C$ such that the quadratic form $\sqrt{-1} \Phi(v,v)$ is indefinite. The aforementioned action is explained in details  in \cite[\S 9]{MR2457528}.
 Note that there is one tautological representation
 \[
 \pi_1^{orb}(\mathbb D^n /\mbox{U}(P,\Phi) ) \simeq \mbox{SU}(P,\Phi)/\{\pm \Id\} \hookrightarrow \PSL(\C) \, ,
 \] 
 which induces for each embedding $\sigma: L \to \C$ one tautological representation  $\pi_1^{orb}(\mathbb D^n /\mbox{U}(P,\Phi) ) \to \PSL(\C)$.
The quotients $\mathbb D^n /\mbox{U}(P,\Phi)$ are always  quasiprojective
orbifolds, and when $[L:\mathbb Q] > 2n$ they are projective (i.e. proper/compact) orbifolds.
The archetypical examples satisfying $[L:\mathbb Q] =2n$ are the Hilbert modular orbifolds, which are quasiprojective but
not projective.

\subsection{Representations and factorization}

A crucial point of the proofs of our results consists in applying some results of factorization of representation of fundamental groups.

\begin{dfn}
\label{def factor}
Let $X$ be a (complex) manifold and let $\rho\colon \pi_1(X)\to G$ be a representation. We say that $\rho$ \emph{factors} through a map $f\colon X\to Y$ towards a manifold $Y$ if there exists a representation $\bar \rho\colon \pi_1(Y) \to G$ such that
$$\rho=\bar\rho \circ f_*.$$
Similarly, we say that $\rho$ \emph{factors} through a map $f\colon X\to Y$ towards an orbifold $Y$ if there exists a representation $\bar \rho\colon \pi_1^{orb}(Y) \to G$ such that
$$\rho=\bar\rho \circ f_*.$$
\end{dfn}

A classical question about the representations of fundamental groups of manifolds is the existence of a "universal factor", in the sense of the following definition. Note that the classical definition of Shafarevich morphism deals with (images in $X$ of) proper normal complex spaces instead of algebraic subvarieties.
\begin{dfn}
\label{dfn Shafarevich}
Let $X$ be a smooth quasi-projective variety and let $\rho \colon \pi_1(X)\to G$ be a representation which factors through an algebraic morphism $f\colon X \to Y$. We say that $f$ is the \emph{Shafarevich morphism} associated to $\rho$ if, for any normal connected algebraic subvariety $Z\subset X$, we have the equivalence
$$\rho(\pi_1(Z))=\{e\} \qquad \Leftrightarrow \qquad f(Z)=\{pt.\}.$$
\end{dfn}

Remark that, if it exists, the Shafarevich morphism associated to a representation is unique.

\subsection{Lifting pseudo-automorphisms}

In order to get rid of orbifold points in factorizations of representations, we will need some results of lifting of pseudo-automorphisms to finite \'etale covers.

\begin{lemma}
\label{finitely generated group}
Let $G$ be a finitely generated group and $H\leq G$ a finite index subgroup. Then there exists a finite index subgroup $H'\leq H$ such that $\phi(H')=H'$ for all $\phi\in \Aut(G)$.
\end{lemma}
\begin{proof}
Denote by $i_H=[G:H]$ the index of $H$ in $G$. Let 
$K=\bigcap_{g\in G}gHg\inv$
be the normal core of $H$. Then $K$ is a normal, finite index subgroup of $H$; more precisely, its index is $i_K=[G:K]\leq i_H^{i_H}$.

Remark that, for a fixed $i$, there are only finitely many normal subgroups of $G$ with index $\leq i$: indeed, such a subgroup can be identified with the kernel of a morphism $G\to G_i$, where $G_i$ is a finite group of cardinality $\leq i$. Since there exist only finitely many such groups $G_i$ and since $G$ is finitely generated, there exist only finitely many equivalence classes of such morphisms, hence only finitely many normal subgroups of $G$ with index $\leq i$.

Fix $i=i_H^{i_H}$, let $\mc S_i$ denote the finite set of normal subgroups of $G$ with index $\leq i$, and let
$$H'=\bigcap_{G'\in \mc S_{i}} G'.$$
Then $H'$ is a normal subgroup of $G$ with finite index, $H' \subseteq H$, and, since every automorphism $\phi$ of $G$ fixes the set $\mc S_i$, a fortiori we have $\phi(H')=H'$.
\end{proof}

\begin{cor}
\label{lift aut}
Let $X$ be a quasi-projective complex manifold and $\nu\colon X'\to X$ be a finite \'etale cover.\\
Then there exists a finite \'etale cover $\eta \colon X''\to X'$ of $X'$ such that every pseudo-automorphism of $X$ lifts to a pseudo-automorphism of $X''$.
\end{cor}

\begin{proof}
Let $G=\pi_1(X)$ and $H=\pi_1(X')$; since $X$ is quasi-projective, $G$ is finitely generated. The injection $\nu_*\colon H\to G$ allows to identify $H$ with a finite index subgroup of $G$; by Lemma \ref{finitely generated group}, we can find a finite index subgroup $H'\leq H$ which is stable by all automorphisms of $G$. Let $\eta\colon X''\to X'$ be the \'etale finite cover corresponding to the inclusion $H'\leq H$ and let
$$\pi=\nu\circ \eta \colon X'' \to X.$$

Now let $f\colon X\rato X$ be a  pseudo-automorphism; let $U=dom(f)$ be the domain of $f$ and $X''_U=\pi\inv(U)$ be the inverse image of $U$. Then the composition
$$f\circ \pi\colon X''_U\to X$$
lifts to a (rational) morphism $X''_U \to X''$ if and only if
$$(f\circ \pi)_* \pi_1(X''_U) \subset \pi_* \pi_1(X'')=H'.$$

Remark that, if $W\subset Y$ is an analytic subset of a complex manifold $Y$ whose complement $Y\setminus W$ has codimension $\geq 2$, then $\pi_1(Y)\cong \pi_1(W)$. More accurately, if $q\in W$, the inclusion $i\colon W \hookrightarrow Y$ induces an isomorphism of fundamental groups
$$i_*\colon \pi_1(W,q)\xrightarrow{\sim} \pi_1(Y,q).$$
This implies that the inclusion induces an isomorphism $\pi_1(U)\cong \pi_1(X)$. Similarly, if $V\subset X$ denotes the domain of $f\inv$, the inclusion induces an isomorphism $\pi_1(V)\cong \pi_1(X)$.

Now, it is not hard to see that the composition
$$\pi_1(X) \cong \pi_1(U) \overset{f_*}{\to} \pi_1(X) \cong \pi_1(V) \overset{(f\inv)_*}{\to} \pi_1(X)$$
is the identity morphism; this means that $f$ induces an automorphism of $\pi_1(X)$.\\
Therefore,
$$(f\circ \pi)_* \pi_1(X''_U) =f_* H' \subset \pi_* \pi_1(X'')=H',$$
which concludes the proof.
\end{proof}

\begin{cor}
\label{lift psaut structure}
Let $\F$ be a codimension $1$ foliation on a smooth projective manifold $X$; assume that $\F$ admits a transverse hyperbolic or projective structure and let
$$G\leq \Psaut(X,\F)$$
be the subgroup of pseudo-automorphisms of $X$ which preserve $\F$ and its transverse structure.\\
Let $X_0=X\setminus H$ be the smooth locus of the structure and let $X_0'\to X_0$ be a finite \'etale cover. Then, after possibly replacing $X_0'$ by a finite \'etale cover, all elements of $G$ lift to pseudo-automorphisms of $X_0'$.
\end{cor}

\begin{proof}
By Corollary \ref{lift aut}, we only need to show that a pseudo-automorphism of $X$ which preserves $\F$ and its transverse structure restricts to a pseudo-automorphism of $X_0$. In order to prove this, one needs to check that if $D\subset X$ is a hypersurface which is not contained in $H$, then the strict transform $f(D)$ (which is a hypersurface because $f$ does not contract any divisor) is not contained in $H$.\\
Indeed, if $p\in D$ denotes a point where $f$ is well-defined and a local isomorphism, the push-forward by $f$ defines a transverse projective structure for $\F$ at a neighborhood of $f(p)$; since we assumed that the transverse structure is preserved, this implies that $f(p)\notin H$.
\end{proof}

Remark that the uniqueness assumption is automatically satisfied in the hyperbolic case (see Remark \ref{hyp->pseff}); in the general case, one needs to impose that $\F$ doesn't come from a foliation defined by a closed rational form (see Lemma \ref{uniqueness projective structure}).

\section{The transversely hyperbolic case}
\label{sec:hyperbolic}

Throughout this section, we denote by $\F$ a foliation admitting a (branched) transverse hyperbolic structure, by $H\subset X$ the hypersurface along which the structure degenerates, namely the support of the negative part of $c_1(N_\F^*)$, and by $X_0:=X\setminus H$ the regular locus of the structure.  The monodromy of the structure is a homomorphism
$$\rho\colon \pi_1(X_0)\to \PSL_2(\R).$$

\subsection{Finiteness of the transverse action}

The goal of this section is to prove the following:

\begin{thm}
\label{hyperbolic transverse action}
Let $X$ be a projective manifold and let $\F$ be a transversely hyperbolic codimension $1$ foliation. Then
\begin{itemize}
\item there exists a generically finite morphism $f\colon X' \to X$ (which is finite \'etale over $X_0$) and a fibration $\pi\colon X\to B$ onto a projective variety $B$ of general type such that $f^*\F$ is the pull-back of a foliation on $B$;
\item the transverse action of $\Psaut(X,\F)$ is finite.
\end{itemize}
\end{thm}

\begin{proof}
As we saw in Remark \ref{hyp->pseff}, the existence of a hyperbolic structure on $\F$ implies that the conormal bundle $N^*_\F$ is pseudo-effective. Therefore, by \cite{MR3644247} (one needs to combine Theorem 1, Proposition 4.6 and Theorem 4 of \emph{loc.cit.} and remark that we are in the case $\epsilon =1$) we have two (non-mutual) possibilities:
\begin{enumerate}
\item either $\F$ is algebraically integrable;
\item or there exists a morphism
$$\Psi \colon X\to \mf H=\mb D^N/\Gamma$$
such that $\F=\Psi^*\G_i$, where $\G_i$ denotes one of the modular foliations on $\mf H$.
\end{enumerate}

First, we may assume that the image by $\rho$ of $\pi_1(X_0)$ is torsion-free. Indeed, by Selberg's lemma this is true for a finite index subgroup $G$ of $\pi_1(X_0)$; replace $X_0$ by its finite \'etale cover $X_0' \to X_0$ corresponding to $G$. The pull-back foliation $\F'$ on $X_0'$ is naturally endowed with a transverse hyperbolic structure, whose monodromy identifies with the restriction of $\rho$ to $G$.\\
By \cite[Theorem 1]{MR3124741} the transverse hyperbolic structure is unique, so that in particular it is preserved by $\Psaut(X,\F)$. Therefore, by Corollary \ref{lift psaut structure}, after possibly taking another finite \'etale cover, all elements of $\Psaut(X, \F)$ lift to pseudo-automorphisms of $X_0'$; of course, the lifts preserve $\F'$ and its transverse hyperbolic structure.\\
Let $X'$ be a smooth compactification of $X_0'$; in order to show the claim for the pair $(X,\F)$, it suffices to show it for the pair $(X', \F')$. Therefore, from now on we will suppose that the monodromy of the structure is torsion-free.

Let $T$ be a current which defines the transverse hyperbolic structure and let $\Theta_T=-T-[N]$ be the associated curvature current. Then $-\Theta_T$ is an $\F$-invariant closed positive current (which represents $c_1(N^*_\F)$), and by \cite[Proposition 2.10(vi)]{MR3124741} the negative part $\{N\}\in H^{1,1}(X,\R)$ is rational. This implies that the monodromy of the structure around the components of $H$ is finite, hence trivial since we assumed that the monodromy is torsion-free.\\
Therefore, by the Riemann extension theorem, a distinguished first integral defined in a small open set in the complement of $H$ extends through $H$, meaning that the representation $\rho$ actually factors through $\pi_1(X)$.

Let us treat first the case where $\F$ is algebraically integrable, or, equivalently, the monodromy $\Gamma=\rho(\pi_1(X)) \leq \PSL_2(\R)$ of the transverse structure is discrete and cocompact (see \cite[Proposition 4.6]{MR3644247}). Let $\widetilde X\to X$ be the universal cover and
$$\tilde \pi \colon \widetilde X \to \mb D$$
be the developing map. By \cite[Theorem 3.2 and \textsection 3.2]{MR3644247} $\tilde \pi$ is surjective and has connected fibres. The fibration obtained by quotient
$$\pi\colon X \to C:=\mb D/\Gamma$$
defines the foliation $\F$.\\
The curve $C$ is uniformized by the disk, therefore it is of general type. In order to conclude, it suffices to remark that elements of $\Psaut(X,\F)$ preserve $\pi$ by definition, and the action on $C$ is identified with the transverse action on $\F$; since the group of automorphisms of a curve of general type is finite, the claim is proved.

%
%
%
%

From now on suppose that we are in the second case: there exists a morphism
$$\Psi \colon X\to \mf H=\mb D^N/\Gamma$$
such that $\F=\Psi^*\G_i$, where $\G_i$ denotes one of the modular foliations on $\mf H$. The monodromy $\rho$ factors through $\Psi|_{X_0}$.

As before, we may assume that $\Gamma$ is torsion-free: indeed, by Selberg's lemma there exists a finite index subgroup $\Gamma'\leq \Gamma$ which is torsion-free.\\
 Replace $X_0$ by its finite \'etale cover $e\colon X_0' \to X_0$ corresponding to the finite index subgroup $\Psi_*\inv(\Gamma')\leq \pi_1(X_0)$. By Corollary \ref{lift psaut structure}, up to taking another finite \'etale cover all elements of $\Psaut(X,\F)$ lift to pseudo-automorphisms of $X_0'$. If $X'$ denotes a smooth compactification of $X_0'$ such that $e$ extends through $X'\setminus X_0'$, we can reason on the pull-back foliation $e^*\F$ on $X'$.

If we denote by
$$X  \overset{\pi}{\to} B \to Z =\im(\Psi) \subset \D^N/\Gamma$$
the Stein factorization of $\Psi$, $\pi$ is $\Psaut(X, \F)$-equivariant: indeed, if $F\subset X$ is a fibre of $\pi$ and $f\in \Psaut(X,\F)$, the algebraic subvariety 
$$\Psi(f(F))\subset \mb D^N/\Gamma$$
is $\G_i$-invariant, and by \cite[Proposition 3.4]{rousseautouzet15} it is reduced to a point. This proves that $\Psaut(X,\F)$ acts by pseudo-automorphisms on $B$.

Since the quotient $\mb D^N/\Gamma$ is smooth and the subvariety $Z\subset \mb D^N/\Gamma$ is compact, by \cite{MR3859271} $Z$ has big cotangent bundle; therefore, by \cite{MR3449168}, $Z$ is a projective variety of general type, hence so is $B$ by pull-back of canonical forms.\\
Since the group of birational transformations of a variety of general type is finite, a finite index subgroup $G\leq \Psaut(X,\F)$ fixes each fibre of $\pi$.

Let $\G$ be the pull-back foliation of $\F_i|_{Z}$ on $B$; we have shown that $\F=\pi^*\G$ and that $B$ is of general type. Furthermore, the finite index subgroup $G\leq \Psaut(X,\F)$ preserves each fibre of $\pi$, hence in particular each leaf of $\F$. This concludes the proof.
\end{proof}

\subsection{Entire curves and special manifolds}

\begin{thm}\label{degen}
Let $X$ be a projective manifold and $\F$ a transversely hyperbolic foliation of codimension $1$ on $X$.
Then any entire curve $f: \C \to X$ is algebraically degenerate i.e. $f(\C)$ is not Zariski dense.
\end{thm}

\begin{rem}
If $X$ is not projective, the statement is false as we will see below in the example of Inoue surfaces which always admit Zariski dense entire curves.
\end{rem}

We shall start with a lemma.

\begin{lemma}
Let $X$ be a projective  manifold and $\F$ a transversely hyperbolic foliation of codimension $1$ on $X$.
Then any entire curve $f: \C \to X$ is tangent to $\F$.
\end{lemma}

\begin{proof}
Let us denote by $h$ the transverse metric which is a smooth transverse metric of constant curvature $-1$ on $X\setminus (\Sing(\F) \cup H)$ (where $H$ is the degeneracy locus of the metric).
Suppose $f: \C \to X$ is not tangent to $\F$. In particular, $f(\C) \not\subset \Sing(\F) \cup H$. Therefore $f^*h$ induces a non-zero singular metric $\gamma(t)=\gamma_0(t)i\,dt\wedge d\overline{t}$ on $\C$ where $\log \gamma_0$ is subharmonic and $\Ric \gamma \geq \gamma$ in the sense of currents. But the Ahlfors-Schwarz lemma (see \cite{De97}) implies that $\gamma \equiv 0$, a contradiction.
\end{proof}

Now, we can prove the theorem.

\begin{proof}
From the preceding lemma, we can suppose that $f: \C \to X$ is tangent to $\F$.
From the study of transversely hyperbolic singular foliations \cite{MR3644247}, we have two cases:
either $\F$ is a fibration, and all leaves are algebraic, or $\F$ is obtained as the pull-back $\Psi^*\G$
where $\Psi$ is a morphism of analytic varieties between $X$ and the quotient $\mf H=\D^n/\Gamma$ of a polydisk, by an irreducible lattice $\Gamma \subset (\Aut \D)^n$ and $\G$ is one of the tautological foliation.
Therefore $\Psi(f): \C \to \mf H$ is tangent to $\G$ and is constant thanks to the hyperbolicity of the leaves on $\mf H$ \cite{rousseautouzet15}. This concludes the proof.
\end{proof}

It seems interesting to relate the above Theorem \ref{degen} to the theory of \emph{special} manifolds as introduced by Campana (see \cite{Ca04} for definitions of special manifolds and conjectures around). 

Campana has conjectured that special manifolds correspond to projective varieties admitting a Zariski dense entire curve. In particular, Theorem \ref{degen} suggests the following question.

\begin{que}
Let $X$ be a projective manifold and $\F$ a transversely hyperbolic (singular) foliation of codimension $1$ on $X$. Prove that $X$ is not special.
\end{que}

The above results also suggest to characterize special manifolds in terms of exceptional locus as Lang's conjectures for general type varieties \cite{Lan86}. 

Let $\Exc(X) \subset X$ denote the Zariski closure of the union of the images of all non-constant holomorphic maps $\C \to X$.

\begin{conj}[Lang]
Let $X$ be a complex projective manifold. Then $X$ is of general type if and only if $\Exc(X) \neq X$.
\end{conj}

Let $X$ be a projective manifold and consider $X_1:=\Pj(T_X)$ the projectivized tangent bundle. All entire curves $f: \C \to X$ can be lifted as entire curves $f_{[1]}: \C \to X_1$. Now, we define an exceptional locus in $X_1$ as: $\Exc_1(X) \subset X_1$ is the Zariski closure of the union of all the images of lifted entire curves $f_{[1]}(\C)$. 

We propose the following conjecture which generalizes Lang's conjecture to the setting of special manifolds.
\begin{conj}
Let $X$ be a projective manifold. Then $X$ is not special if and only if $\Exc_1(X) \neq X_1$.
\end{conj}

This suggests the following question.

\begin{que}
Let $X$ be a projective manifold and $\F$ a holomorphic foliation on $X$ such that all entire curves are tangent to $\F$. Is it true that all entire curves in $X$ are algebraically degenerate and that $X$ is not special ?
\end{que}

More generally, one may consider inductively jets spaces $\pi_{0,k}: X_k \to X $ (see \cite{De97}) and the corresponding exceptional loci $\Exc_k(X)$ obtained as  the Zariski closure of the union of all the images of lifted entire curves $f_{[k]}(\C) \subset X_k$. Then we ask the following question.

\begin{que}
Let $X$ be a projective manifold. Suppose there is an integer $k\geq 0$ such that $\Exc_k(X) \neq X_k$. Is it true that all entire curves in $X$ are algebraically degenerate and that $X$ is not special ?
\end{que}

This question is also motivated by recent results of Demailly \cite{De11} and Campana-P\u{a}un \cite{CaPa15} which imply the following weaker statement: $X$ is of general type if and only if there is an integer $k\geq 1$ such that $\Bs(\mathcal{O}_{X_k}(m)\otimes \pi_{0,k}*A^{-1}) \neq X_k$, where $A$ is an ample line bundle on $X$. The relationship with the previous question is clear with the now classical fact that $\Exc_k(X) \subset \Bs(\mathcal{O}_{X_k}(m)\otimes \pi_{0,k}*A^{-1})$ (see \cite{De97}).

\subsection{A transversely hyperbolic foliation with infinite transverse action}
In this section, we will see that if the K\"ahler assumption is dropped, one can construct transversely hyperbolic foliations with non finite transverse action and Zariski dense entire curves.

More precisely, let us consider Inoue surfaces \cite{Ino} which are quotients of $\Hj \times \C$, where $\Hj$ is the upper half-plane, by certain infinite discrete subgroups. They are equipped with a natural transversely hyperbolic foliation.There are three type of Inoue surfaces distinguished by the type of their fundamental group: $S_M$, $S^{(+)}$ and $S^{(-)}$.

Let us describe the Inoue surfaces of type $S_M$. Let $M=(m_{i,j}) \in \SL(3, \Z)$ be a unimodular matrix with eigen-values $\alpha, \beta, \overline{\beta}$ such that $\alpha>1$ and $\beta \neq \overline{\beta}$. We choose a real eigen-vector $(a_1, a_2, a_3)$ and an eigen-vector $(b_1, b_2, b_3)$ of $M$ corresponding to $\alpha$ and $\beta$. Let $G_M$ be the group of analytic automorphisms of $\Hj \times \C$ generated by
\begin{itemize}
\item $g_0: (w,z) \to (\alpha w, \beta z)$
\item $g_i: (w,z) \to (w+a_i, z+b_i)$
for $i=1,2,3.$
\end{itemize}
$S_M$ is defined to be the quotient surface $\Hj \times \C/G_M$.

Consider the automorphisms of $\Hj \times \C$, $(w,z) \to (n w,n z)$. They induce automorphisms $h_n$ of $S_M$
which have infinite transverse action provided $n \neq \alpha^{l/p}$ for $l, p$ integers. One should also remark that 
in such surfaces all (non-constant) entire curves are tangent to the foliation and are Zariski dense (their topological closure is a real torus of dimension $3$).

Here, the representation $\rho_\F:\pi_1 (S_M)\to \PSL(2,\R)$  associated to the transverse hyperbolic structure takes values in the affine subgroup $\Aff(2,\R)$ of and its linear part $\rho_\F^1:\pi_1 (S_M)\to ({\R}_{>0},\times)$ has non trivial image.

It is worth noticing that this situation cannot occur in the K\" ahler realm. Indeed, suppose that $X$ is a compact K\"ahler manifold carrying a transversely hyperbolic  codimension one foliation which is also transversely affine and such the linear part  $\rho_\F^1:\pi_1 (X)\to ({\R}_{>0},\times)$ has non trivial image. To wit, there exists on $X$ an open cover $(U_i)$ such that for every $i$, $\F$ is defined by $dw_i=0$, where $w_i:U_i\to\Hj$ is submersive on $U_i -\mbox{Sing}\ \F$ and such that the glueing conditions $w_i=\varphi_{ij}\circ w_j$ are defined by locally constant elements $\varphi_{ij}$ of $\Aff(2,\R)$. In particular there exists locally constants cocycles $a_{ij}\in{\R}_{>0}$ such that 

\begin{equation} \label{E:locconstantglueing}
dw_i=a_{ij}dw_j
\end{equation}
 and the normal bundle $N_\F$ is thus numerically trivial. On the other hand, the existence of a transverse hyperbolic structure directly implies that  $N_\F$ is equipped with a metric whose curvature  is a non trivial  semi-negative form. This shows that $c_1(N_\F)\not=0$, a contradiction.

%


\section{The transversely projective case}
\label{sec transv proj}

Throughout this section we let $X$ be a projective  manifold, $\F$ be a transversely projective foliation of codimension $1$ on $X$ and $\Psaut(X,\F)\leq \Psaut(X)$ be the group of pseudo-automorphisms which preserve $\F$. 

Denote by $E\rato X$ any rank two vector bundle such that the given projective structure on $\F$ is defined by a Riccati foliation $\mc R$ on $\Pj(E)$. The foliation $\mc R$ is defined by a (non-unique) flat meromorphic connection $\nabla$ on $E$, which induces a monodromy representation
$$\rho_\nabla \colon \pi_1(X\setminus (\nabla)_\infty) \to \SL_2(\C),$$
where $(\nabla)_\infty$ denotes the divisor of poles of $\nabla$.\\
The monodromy representation of the projective structure is a representation
$$\rho \colon \pi_1(X_0 ) \to \PSL_2(\C),$$
where $X_0:=X\setminus$ and $H\subset (\nabla)_\infty$ denotes the divisor along which the projective structure degenerates. Such representation is induced by $\rho_\nabla$ upon projectivization.

\begin{proof}[Proof of Theorem \ref{thm transv proj}]
By \cite[Theorem D]{MR3522824}, at least one of the following is true:
\begin{enumerate}
\item there exists a generically finite  morphism $\pi \colon X'\to X$ such that $\pi^* \F$ is defined by a closed rational $1$-form;
\item there exists a rational dominant map $\eta \colon X \rato S$ to a ruled surface $\pi\colon S\to C$ and a Riccati foliation $\mc G$ defined on $S$ (i.e. over the curve $C$) such that $\F=\eta^* \mc G$;
\item there exists a polydisk Shimura modular orbifold $\mathfrak H=\mb D^N/\Gamma$ and an algebraic map $\psi\colon X_0 \to  \mathfrak H$ such that the monodromy representation $\rho$ factors through one of the tautological representations of $\pi_1(\mathfrak H)$ (up to a field automorphism of $\C$). Furthermore the singularities of the transverse projective structure are regular.
\end{enumerate}

In the first case, we fall into the first alternative of the statement; the second and the third cases are settled by Proposition \ref{factorization curve} and Proposition \ref{factorization Shimura} respectively whose proofs are given below.
\end{proof}

\subsection{The case of surfaces}

In this section we treat the case of birational symmetries of foliations of surfaces. Most of the key results are contained in \cite{MR1998612}. We want to prove the following:

\begin{prop}
\label{factorization curve}
Suppose that there exist a dominant rational map $\eta \colon X \rato S$ towards a surface $S$ and a foliation $\mc G$ on $S$ such that $\F=\eta^*\mc G$. Then
\begin{itemize}
\item either there exists $\pi \colon X'\to X$, where $\pi$ is  generically finite, such that $\pi^* \F$ is defined by a closed rational $1$-form;
\item or the transverse action of $\Psaut(X,\F)$ is finite.
\end{itemize}
\end{prop}

We start by the following lemma, which is well-known to specialists.

\begin{lemma}
\label{transverse symmetries}
Let $\G$ be a codimension one foliation on a complex manifold $Y$. Suppose that $\G$ is invariant by the flow of a vector field $v$ on $Y$ which is not everywhere tangent to $\G$; then $\G$ is defined by a closed rational $1$-form.
\end{lemma}

\begin{proof}
Let $\omega$ be a rational form on $Y$ defining $\G$; for example, one can take a form with values in $N^*_\F$ which defines $\G$, and divide it by any non-zero meromorphic section of $N^*_\F$.

Let us show that the rational form $\tilde \omega:=\omega/\omega(v)$ is closed. It is enough to check that $d\tilde \omega=0$ at smooth points of $\F$ such that $v$ is locally transverse to $\F$; in a neighborhood of such a point we can find local coordinates $(z,w_1,\ldots ,w_n)=(z,w)$ such that
$$\omega=a(z,w)\, dz, \qquad v=b(z,w) \frac{\partial}{\partial z}.$$
The condition that the flow of $v$ preserves $\F$ means that $b(z,w)=b(z)$ does not actually depend on $w$. Therefore $\tilde w =  dz/b(z)$ is closed, which concludes the proof.
\end{proof}

The proof of the following lemma is essentially contained in \cite{MR3294560}, but we prove it again for the convenience of the reader.

\begin{lemma}
\label{abelian monodromy}
Let $\F$ be a transversely projective foliation with abelian monodromy and at worst logarithmic singularities. Then $\F$ is defined by a closed rational $1$-form.
\end{lemma}

\begin{proof}
Remark that abelian subgroups of $\PSL_2(\C)$ are conjugated to subgroups of $(\C,+)$ or of $(\C^*,\times)$. Therefore, we may assume that the monodromy is either additive or multiplicative.

If the monodromy is additive, then the local distinguished first integrals $F_i$ can be chosen so that the local meromorphic forms $dF_i$ glue to a closed rational form which is defined outside the singularities of the structure and which defines $\F$.\\
Similarly, if the monodromy is multiplicative, then the $F_i$-s can be chosen so that the local meromorphic forms $dF_i/F_i$ glue to a closed rational form which is defined outside the singularities of the structure and which defines $\F$.

The assumption on the singularities ensures that such forms can be extended meromorphically through them. By GAGA, the meromorphic forms obtained in this way are rational.
\end{proof}

We are ready to prove Proposition \ref{factorization curve}.

\begin{proof}[Proof of Proposition \ref{factorization curve}]
First, remark that we can replace $S$ by the Stein factorization $\pi\colon S' \to S$ of $\eta$ and $\mc G$ by $\pi^*\mc G$. Therefore, we may suppose that the fibres of $\eta$ are connected.

Let us prove that
\begin{itemize}
\item either the action of $\Psaut(X,\F)$ on $X$ preserves $\eta$ 
\item or $\F$ is algebraically integrable (and in particular it is defined by a closed rational form).
\end{itemize}
 Let us fix $f\in \Psaut(X,\F)$ and suppose that for a fibre $Y$ of $\eta$ we have $\eta(f(Y))\neq \{\text{pt.}\}$. Since $Y$ is $\F$-invariant and $f$ preserves the foliation $\F$, $f(Y)$ is also $\F$-invariant, thus $\eta(f(Y))$ is $\mc G$-invariant; in particular, since $\eta(f(Y))$ has dimension $\geq 1$, this means that $\eta(f(Y))$ is a $\G$-invariant algebraic curve. \\
Remark that, if $\eta(f(Y))\neq \{\text{pt.}\}$, then the same is true for nearby fibres. However, by \cite{MR531151,MR1753461} if there is an infinite number of $\G$-invariant curves then $\G$ is algebraically integrable, therefore so is $\F$. This shows the alternative.

From now on, we suppose that the action of $\Psaut(X,\F)$ preserves $\eta$, meaning that $\eta$ induces a group homomorphism
$$\phi\colon \Psaut(X,\F) \rato \Bir(S,\G).$$
In order to show that the transverse action of $\Psaut(X,\F)$ on $\F$ is finite, it is enough to show that the transverse action of $\Bir(S,\G)$ on $\G$ is finite. Suppose that this is not the case, so that in particular $\Bir(S,\G)$ is infinite.

By \cite[Theorem 1.3]{MR1998612} if for every birational model $(S',\G')$ of $(S,\G)$ we have $\Aut(S',\G')\subsetneq \Bir(S',\G')$, then either $\G$ is algebraically integrable or $(S,\G)$ is birationally equivalent to one of the two situations in Example 1.3 of \emph{loc.cit.}: after possibly pulling back by a generically finite  morphism, $S=\Pj^1\times \Pj^1$ and $\G$ is defined by a form written as $\alpha y\, dx + \beta x \, dy$, whose multiple
$$\alpha \frac{dx}x + \beta \frac{dy}y$$
is a closed rational form defining $\G$.\\
Therefore, we may assume that $\Aut(S,\G)=\Bir(S,\G)$ is an infinite group. 

By \cite[Proposition 3.9]{MR1998612} at least one of the following is verified:
\begin{enumerate}
\item $\Aut(S,\G)$ contains an element of infinite order $f$ whose action on $H^{1,1}(S,\mb R)$ satisfies 
$$||(f^n)^*|| \to +\infty \qquad \text{as }n\to +\infty;$$
\item $\Aut(S,\G)$ contains the flow of a vector field $v$ on $S$
\end{enumerate}

In the first case, by  \cite[Theorem 3.1, Theorem 3.5]{MR1998612} there exists a generically finite  morphism $\nu \colon S' \to S$ such that $\nu^* \G$ is either a linear foliation on an abelian surface or an elliptic fibration; in both cases, $\nu^* \G$ is defined by a closed rational $1$-form, therefore so is $\F$ after pull-back by a generically finite  morphism $\pi\colon X'\to X$ induced by $\nu$.

In the second case, by \cite[Proposition 3.8]{MR1998612} we are in one of the following situations:
\begin{itemize}
\item $v$ is tangent to an elliptic fibration and $\G$ is either a turbolent foliation (so that we can apply Lemma \ref{transverse symmetries}) or the fibration itself. In both cases, $\G$ is defined by a closed rational $1$-form.
\item $\G$ is a linear foliation on a torus, thus it is defined by a closed regular $1$-form.
\item $v$ is tangent to a fibration in rational curves and $\G$ is either a Riccati foliation (so that we can apply Lemma \ref{transverse symmetries}) or the fibration itself. In both cases, $\G$ is defined by a closed rational $1$-form.
\item $S$ is a $\Pj^1$-bundle over an elliptic curve $E$, $v$ projects onto a vector field on $E$ and $\G$ is either obtained by suspension or it coincides with the $\Pj^1$-bundle (so that in particular it is algebraically integrable). In the first case $\G$ is smooth and the $\Pj^1$-bundle induces a transverse projective structure without poles, whose monodromy factors through $E$; in particular the monodromy is abelian, which implies that $\G$ is defined by a closed rational form by Lemma \ref{abelian monodromy}. Therefore, in both cases $\G$ is defined by a closed rational form.
\item Up to a change of birational model, $\G$ is a linear foliation on $\Pj^1_x\times \Pj^1_y$, which means that it is defined by a closed rational form of type
$$\alpha \frac{dx}x + \beta \frac{dy}y.$$
\end{itemize}
This proves that if a foliation on a surface is preserved by a holomorphic vector field, then it is defined by a closed rational $1$-form, which concludes the proof.
\end{proof}

\subsection{Factorization through a Shimura modular orbifold}


Throughout this section, suppose that there exists an algebraic quotient of the polydisk $\mf H=\mb D^N/\Gamma$ and an algebraic map 
$$\psi\colon X_0 \to \mb \mf H$$
such that the monodromy representation $\rho$ factors through one of the tautological representations of $\pi_1^{orb}(\mf H)$ (up to a field automorphism of $\C$).

Our goal is to prove the following:

\begin{prop}
\label{factorization Shimura}
Under the assumption above, at least one of the following is verified:
\begin{itemize}
\item either there exists a generically finite  morphism $\pi \colon X'\to X$ such that $\pi^* \F$ is defined by a closed rational $1$-form;
\item or the transverse action of $\Psaut(X,\F)$ is finite.
\end{itemize}
\end{prop}

\begin{lemma}
\label{Shimura Shafarevich}
Suppose that $\mb D^N/\Gamma$ is smooth. Then the Stein factorization of $\psi \colon X_0 \to \im(\psi)$ is the Shafarevich morphism of the monodromy representation $\rho$ (in the sense of Definition \ref{dfn Shafarevich}).
\end{lemma}

\begin{proof}
Let $\psi_0\colon X_0 \to B$ be the Stein factorization of $\psi\colon X_0\to \im (\psi)$. We already know that the monodromy is trivial along fibres; what is left to prove is that $\psi_0$ is "maximal" among such morphisms, i.e. that, if $i\colon Y\hookrightarrow X_0$ is a normal connected algebraic subvariety such that the image of $\rho\circ i_* \colon \pi_1(Y) \to \PSL_2(\C)$ is finite, then $\psi_0(Y)$ is reduced to a point. 

Suppose by contradiction that $\psi(Y)$ is not reduced to a point and let $Z=\overline{\psi(Y)}\subset \mf H$; by \cite[Proposition 3.4]{rousseautouzet15}, $Z$ is not tangent to any of the $\F_i$.

The leaves of any one of the $\F_i$ are locally given by first integrals with values in $\mb D$; since up to composing with an element of $\Gal(\C/\Q)$ the representation $\rho$ factors through the monodromy of the natural transverse hyperbolic structure on $\F_i$, the latter is finite along $Z$.

Consider the finite \'etale cover $Z'\to Z$ corresponding to the finite index subgroup $\ker(\rho\circ i_*) \subset \pi_1(Y)$. Then the pull-back of $\F_i$ to (the smooth part of) $Z'$ is given by a global first integral $Z'_{sm}\to \mb D$; by the Riemann extension theorem, such first integral extends holomorphically to a smooth projective model of $Z'$, hence it is constant, a contradiction. This shows that the factorization of $\psi$ is indeed the Shafarevich morphism of $\rho$.
%
%
\end{proof}

The following lemma follows from the proof of \cite[Lemma 2.20]{MR2324555}; one can actually see that the morphism $\pi$ in the statement can be taken to have topological degree $\leq 2$.

\begin{lemma}
\label{uniqueness projective structure}
If $\F$ admits more than one transverse projective structure, then there exists $\pi \colon X'\to X$, where $\pi$ is  generically finite, such that $\pi^* \F$ is defined by a closed rational $1$-form.
\end{lemma}

We are ready to prove Proposition \ref{factorization Shimura}.

\begin{proof}[Proof of Proposition \ref{factorization Shimura}]
Suppose that there exists no generically finite  morphism $\pi \colon X'\to X$ such that $\pi^* \F$ is defined by a closed rational $1$-form; by Lemma \ref{uniqueness projective structure} this implies that the transverse projective structure of $\F$ is unique, so that in particular it is preserved by $\Psaut(X,\F)$.

First, we may assume that $\Gamma$ is torsion-free: indeed, by Selberg's lemma there exists a finite index subgroup $\Gamma'\leq \Gamma$ which is torsion-free.\\
Replace $X_0$ by its finite \'etale cover $X_0' \to X_0$ corresponding to the finite index subgroup $\psi_*\inv(\Gamma')\leq \pi_1(X_0)$. By Corollary \ref{lift psaut structure}, up to taking another finite \'etale cover, all elements of $\Psaut(X,\F)$ lift to pseudo-automorphisms of $X_0'$, and we can reason on the pull-back foliation on $X_0'$.\\
Remark that, if $X'$ is a smooth compactification of $X_0'$ and $X'\to X$ denotes a generically finite morphism which restricts to the \'etale cover $X_0'\to X_0$, we have $\kappa(X')\geq \kappa(X)\geq 0$; therefore, it suffices to prove the claim for the pull-back of $\F$ on $X'$.

If we denote by
$$X_0  \overset{\pi}{\to} B \to Z \subset \D^N/\Gamma$$
the Stein factorization of $\psi$, $\pi$ is $\Psaut(X, \F)$-equivariant: indeed, by Lemma \ref{Shimura Shafarevich} it is identified with the Shafarevich morphism of the monodromy $\rho$, and therefore it only depends on the transverse projective structure. By uniqueness, it is canonically associated to $\F$.\\
Furthermore, as we saw in the proof of Corollary \ref{lift psaut structure}, elements of $\Psaut(X,\F)$ restrict to pseudo-automorphisms of $X_0$. This proves that $\Psaut(X,\F)$ acts by pseudo-automorphisms on $B$.

Since $\mb D^N/\Gamma$ is smooth, by \cite{MR3859271} the subvariety $Z:= \psi(X_0)\subset \mb D^N/\Gamma$ has big logarithmic cotangent bundle; therefore by \cite{MR3449168} it is of log-general type (see \cite{MR637060}), hence so is $B$ by pull-back of log-canonical forms.\\
By \cite[Theorem 11.12]{MR637060}, the group of  strictly birational transformations of $B$ is finite, and in particular so is its group of pseudo-automorphisms. Therefore, a finite index subgroup $G\leq \Psaut(X,\F)$ fixes each fibre of $\pi$.

If the fibres of $\pi$ are $\F$-invariant, the proof is finished: $\F$ is the pull-back of a foliation $\mc G$ on $B$, and the finite index subgroup $G\leq \Psaut(X,\F)$ fixes each leaf of $\mc G$, thus each leaf of $\F$.

Suppose now that generic fibres of $\pi$ are not $\F$-invariant. The restriction of $\F$ to such fibres is a codimension $1$ foliation endowed with a natural transverse projective structure, which is preserved by the restriction of the action of $G$ to fibres. \\
Since the monodromy of the structure is trivial in small neighborhoods of fibres of $\pi$ (in $X_0$), in such a neighborhood $U$  one can define a first integral
$$F_U\colon U\rato \Pj^1$$
which defines the transverse structure. Since the action of $G$ preserves such structure, for $g\in G$ one has
$$F_U\circ g = \phi_g \circ F_U \qquad \text{for some }\phi_g\in \PSL_2(\C).$$

Let us show that the action of $G$ on $U$ is transversely finite; this is equivalent to showing that the group morphism
$$\Phi\colon G \to \PSL_2(\C)$$
defined by $g \mapsto \phi_g$ has finite image.\\
First remark that, since the singularities of the structure are regular by \cite{MR3522824}, the first integral $F_U$ extends meromorphically to the closure of $U$ in $X$. In particular, if we denote by $Y$ the Zariski-closure in $X$ of a fibre of $\pi$ contained in $U$, the restriction of $\F$ to $Y$ is given by a meromorphic first integral
$$F_Y\colon Y \rato \Pj^1.$$

By the easy addition formula (see e.g. \cite[\textsection 10]{MR637060}), for a general fibre $F$ of $\pi$ we have
$$\kappa(X)\leq \kappa(F) + \dim(B),$$
which implies that $\kappa(F)\geq 0$. Therefore, by \cite{lobiancopadic}, the transverse action of $G$ on $Y$ (i.e. the action of $G$ on the first integral $F_Y$) is torsion; since such action coincides with the action of $G$ on $\Pj^1$ defined above, this implies that the image of $\Phi$ is torsion.\\
Suppose by contradiction that $\Phi(G)$ has infinite order; torsion subgroups of Lie groups are virtually abelian (see e.g. \cite{MR0393342}), and  abelian torsion subgroups of $\PSL_2(\C)$ are conjugated to rational angle rotation subgroups. Since we supposed the image of $\Phi$ to be infinite, the conjugation which puts a finite index subgroup in the form of rotations is unique up to composition with the involution $i\colon z\mapsto z\inv$; therefore, up to composition with $i$, the first integral $F_U$ can be chosen canonically.\\
As remarked above, the morphism $\Phi$ coincides with the transverse action of $G$ on any fibre in $U$ (endowed with the restricted foliation); in particular, if one chooses another small neighborhood of a fibre $U'$ as above such that $U\cap U'\neq \emptyset$, the transverse action of $G$ will also be infinite. As a consequence, on $U'$ one can also define a canonical (up to involution $i$) first integral
$$F_{U'} \colon U' \rato \Pj^1.$$
Since, on $U\cap U'$, $F_U$ and $F_{U'}$ differ at most by composition with $i$, the local rational forms $dF_U/F_U$ glue to a global rational form (on $X_0$) which defines $\F$; by the regularity of singularities of the structure, such form extends to $H$, and we obtain a contradiction with the assumption at the beginning of the proof. This shows that the transverse action of $G$ on $U$ is finite.

Now, consider the foliation $\G$ obtained by intersecting (local) leaves of $\F$ with fibres of $\pi$. Since the restriction of $\F$ to fibres of $\pi$ is algebraically integrable, $\G$ is also algebraically integrable; let 
$$\eta \colon X \rato B_\G$$
 be a rational fibration defining $\G$.\\
Remark that the action of $G$ preserves $\eta$. The fact that $G$ has transversely finite action on $U$ can be rephrased by saying that some finite index subgroup $G'$ of $G$ acts as the identity on $\eta(U)$. Since birational transformations which act as the identity on a euclidean neighborhood are the identity, this implies that $G'$ fixes all fibres of $\eta$, and in particular all leaves of $\F$. This concludes the proof.
\end{proof}

\section{The case of codimension $1$ foliations on tori}
\label{sec tori}

Conjecture \ref{main conjecture} is only meaningful for manifolds with a rich group of automorphisms; the first example one should study is therefore that of homogeneous (compact K\"ahler) manifolds. By a result of Borel and Remmert (see for example \cite[Theorem 2.5]{MR1440947}), such a manifold can be decomposed in a product $T\times R$, where $T=\C^n/\Lambda$ is a complex torus and $R$ is a rational homogenous manifold (a generalized flag manifold); in particular, in order for a homogeneous manifold $X$ to have non-negative Kodaira dimension, $X$ needs to be a torus.

\subsection{Classification}

Codimension $1$ foliations on complex tori have been classified by Brunella:

\begin{thm}[\cite{MR2674768}]
\label{Foliations tori}
Let $\F$ be a (singular) codimension one foliation on a complex torus $X=\C^n/\Lambda$. Then exactly one of the following holds:
\begin{enumerate}
\item $\F$ is a linear foliation;
\item $\F$ is a turbulent foliation: there exists a linear projection $\pi\colon X\to Y$ onto a complex torus $Y=\C^m/\Lambda'$, $0<m<n$, a closed meromorphic one-form $\eta_0$ on $Y$ and a holomorphic (linear) form $\eta_1$ on $X$, which doesn't vanish on the fibres of $\pi$, such that $\F$ is defined by the meromorphic one-form
$$\eta:= \pi^*\eta_0+\eta_1;$$
\item the normal bundle $N_\F$ is ample.
\end{enumerate}
\end{thm}

Remark that $N_\F$ is automatically effective: indeed, the image of a generic vector field on $X$ through the natural projection $T_X\to N_\F$ yields a non-trivial section. One can then construct the \emph{normal reduction} of $\F$ (see \cite{MR2674768}), i.e. a linear projection
$$\pi\colon X\to Y$$
onto a complex torus $Y=\C^m/\Lambda'$ such that $N_\F=\pi^*L$ for some \emph{ample} line bundle $L$ on $Y$. Case $(1)$ and $(3)$ in Theorem \ref{Foliations tori} correspond to the extremal cases $m=0$ and $m=n$ respectively.

Smooth foliations, which had been previously classified by Ghys \cite{MR1440947}, arise exactly in the cases $m=0$ and $m=1$.

\subsection{Symmetries}

Recall that any meromorphic map $X \rato T$ from a compact complex manifold towards a compact complex torus is actually holomorphic (see \cite[Lemma 3.3]{MR0481142}); therefore, when studying Conjecture \ref{main conjecture} on complex tori one can simply study automorphisms preserving the foliation.

\begin{prop}
\label{tori symmetries}
Let $\F$ be a (singular) codimension one foliation on a complex torus $X=\C^n/\Lambda$ and let $G=\Aut(X,\F)$ be the group of symmetries of $\F$.
\begin{enumerate}
\item If $\F$ is a linear foliation, then $G$ contains the group of translations of $X$, which has transversely infinite action on $\F$.
\item If $\F$ is a turbulent foliation, then with the notation of Theorem \ref{Foliations tori} $G$ preserves $\pi$ and its action on $Y$ is finite; $G$ contains the subgroup of translations of $X$ preserving $\pi$, which has infinite transverse action on $\F$.
\item If the normal bundle $N_\F$ is ample, then $G$ is a finite group.
\end{enumerate}
\end{prop}

Remark that in cases $1$ and $2$ the group $G$ can actually be much bigger (for example, it may contain elements with infinite linear action): this depends on resonance conditions between the lattice defining $X$ (respectively, the fibres of $\pi$) and the holomorphic form defining $\F$ (respectively the form $\eta_1$).

\begin{proof}
Suppose first that $\F$ is a linear foliation. Then clearly $G$ contains the group of translations, and we only need to prove that the latter has transversely infinite action on $\F$. If this weren't the case, then $X$ would be covered by a finite union of leaves of $\F$, which contradicts the fact that leaves have zero Lebesgue measure.

Now consider the case of a turbulent foliation defined by a meromorphic one-form
$$\eta=\pi^*\eta_0 + \eta_1,$$
where the same notation as in Theorem \ref{Foliations tori} is used. Since the projection $\pi$ is canonically associated to the normal bundle $N_\F$, hence to the foliation $\F$, $G$ preserves it.\\
Furthermore, the action of $G$ on the base $Y$ preserves the ample line bundle $L$; by \cite[Application 1, section 6]{MR2514037}, such action is finite.\\
It is clear that any translation of $X$ along fibres of $\pi$ preserves $\eta$, hence $\F$. The same argument as in the case of linear foliations allows to conclude that the action of such translations, hence that of $G$, is transversely infinite.

Finally, suppose that $N_\F$ is ample. The same argument as above allows to conclude that $\Aut(X,N_\F)$ is finite, hence so is $G$.
\end{proof}

\begin{rem}
Proposition \ref{tori symmetries} does not contradict Conjecture \ref{main conjecture}. Indeed, if a foliation $\F$ on an abelian variety admits a transverse invariant metric which is constructed as the curvature of an hermitian metric on a line bundle $L$, then the leaves of $\F$ are the fibres of a linear projection $\pi\colon X\to E$ onto an elliptic curve $E$, which is nothing but the normal reduction of $L$; in particular, the action of $\Aut(X,\F,L)=\Aut(X,L)$ on the space of leaves identifies with the action of $\Aut(E,L_E)$ on $E$, where $L_E$ is an ample line bundle such that $L=\pi^*L_E$. Such action is finite by \cite[Application 1, section 6]{MR2514037}.

Let us prove the above claim. Suppose that $\F$ admits an invariant transverse metric $\theta$ which can be constructed as the curvature form $\Theta_h$ of an hermitian metric on a line bundle $L$ on $X$. In particular, this implies that $L$ is pseudo-effective; by \cite[Lemma 1.1]{MR1646050}, $L$ is numerically equivalent to an effective line bundle $L'$. Let
$$\pi\colon X\to E$$
be the normal reduction of $L'$, and let $A$ be an ample line bundle on $E$ such that $L'=\pi^*A$. Remark that, since $0=[\theta]^2=\pi^* c_1(A)^2$, $E$ is one-dimensional.\\
Now, since $[\theta]=\pi^* c_1(A)$, the integral of $\theta$ along any closed curve contained in a fibre of $\pi$ is equal to $0$. This implies that the fibres of $\pi$ coincide with the leaves of $\F$, which concludes the proof.
\end{rem}

\bibliography{references}{}
\bibliographystyle{alpha}

\end{document}